\documentclass[amsart, 12pt]{amsart}

\usepackage{amssymb,amsfonts,amsthm,amsmath,latexsym}
\usepackage{enumerate, caption, array, esint}

\usepackage{pgfplots}
\pgfplotsset{compat=1.15}
\usepackage{mathrsfs}
\usetikzlibrary{arrows}
\usepackage{subcaption}

\usepackage[margin=1.1in]{geometry}

\makeatletter
\newcommand*{\rom}[1]{\expandafter\@slowromancap\romannumeral #1@}
\makeatother

\theoremstyle{plain}
\newtheorem{thm}{Theorem}[section]

\newtheorem{lem}[thm]{Lemma}
\newtheorem{prop}[thm]{Proposition}

\newtheorem{example}[thm]{Example}

\theoremstyle{definition}
\newtheorem{defn}[thm]{Definition}

\theoremstyle{remark}
\newtheorem{rem}[thm]{Remark}

\numberwithin{equation}{section}

\newcommand{\hide}[1]{{}}

\theoremstyle{plain}

\begin{document}

\title{Validated numerics for the Gauss problem on Continued Fractions}

\author{M. Pollicott }

\address{Department of Mathematics, Warwick University, Coventy, CV4 7AL-UK}
\email{masdbl@warwick.ac.uk}

\date{\today}

\maketitle

\section{Introduction}
It is a classical area  of research to  understand  the statistical properties of continued fractions.  One of the earliest and most famous  such problem   was  studied by Gauss and  involved the tails of continued fraction expansions
More precisely, for any irrational $0 < x < 1$ there is an infinite continued fraction expansion
$$
x = \frac{1}{a_1 + \frac{1}{a_2 + \cdots + \frac{1}{a_{n-1} + \frac{1}{a_{n} + \cdots}}}}
$$
which, for any $n \geq 1$,  we can write as 
$$
x = \frac{1}{a_1 + \frac{1}{a_2 + \cdots + \frac{1}{a_{n-1} +\xi_{n}}}}
\hbox{ where }
\xi_n = \frac{1}{a_n + \frac{1}{a_{n+1}+\frac{1}{a_{n+2} + \frac{1}{a_{n+3} + \cdots}}}}
\in (0,1)
$$
In particular, we have the inductive definition $\xi_n =  \{\frac{1}{\xi_{n-1}}\}$ where $\{ \cdot \}$ denotes the fractional part.
The distribution of the sequence $(\xi_n)_{n=1}^\infty$ for a typical $x$ was considered  by Gauss 
and appears as  a note in his mathematical notebook.  On 25 October 1800
he wrote that he had a simple argument that
 if $F_n(x)$  is the probability that  $\xi_n<x$ for a random $x$, then
$$ \lim_{n\to \infty}F_n(x)= \frac{\log (1+x)}{\log 2} \eqno(1)$$
But in a letter to Laplace dated January 30, 1812  Gauss said 
that he had not  solved satisfactorily this problem \cite{Gauss-letter}.
\footnote{To quote from Gauss's  letter: ``Je me rappelle pourtaat d’un probleme curieux duquel je me suis occupd il y a 12 ans, mais lequel je n’ai pas rdussi alors k resoudre k ma satisfaction. Peut etre daignerez-vous en occuper quelques moments: dans ce cas je suis sur que vous trouverez une solution plus complete.''}

The first published proof of (1) appeared in 1928 and was due to Kuzman \cite{Kuzmin}, 
where he actually showed the stronger result  that 
$$F_n(x) = \log (1+x) + O(q^{\sqrt{n}}) \hbox{ as $n \to +\infty$ }$$
 for some $0 < q < 1$
(see \cite{Khintchine} for an  interesting  exposition).
In 1929, P. Levy improved this result to 

$$F_n(x) = \log (1+x) + O(q^n) \hbox{ as $n \to +\infty$ }$$
with $q = 3.5 - 2\sqrt{2} = 0.67157\ldots$.  Subequently, Sz\"usz refined the error to 
$q < 0.4$ \cite{Szusz} and Schweiger \cite{Schweiger} gave a rigorous bound of $q \leq \frac{2}{3+\sqrt{5}} = 0.381966\ldots$.  
  In 1974, Wirsing further improved these estimates  to 
$$
F_n(x) = \log (1+x) + (-\lambda)^n \Phi(x) (1+o(1))
\hbox{ as $n \to +\infty$}
\eqno(2)
$$
where $\lambda = 0.30366 30028 98732 65860 \ldots $ was given to $20$ decimal places 
  \footnote{Mayer quoted this value  from Wirsing's work, but adds that ``the accuracy was not clear''.  In fact, it is correct except for the last digit}
 and $\Phi:[0,1] \to \mathbb R$ is a  real analytic function
 \footnote{This gave a partial  answer to Exercise 22 on page 337 of the first edition of \cite{Knuth} about the convergence.}
  \cite{Wirsing}.  
 
 The value of $\lambda$  plays a role in estimates on the running time of the Euclidean Algorithm (cf.  \cite{Knuth})
The numerical value of $\lambda$ was given  to 35 decimal places
\footnote{Although they presented an estimate accurate to $35$ decimal places they only expressed confidence in the value to $30$ decimal places}
 by Flajolet and Vallée 
\cite{Flajolet}, \cite{Flajolet2}
in 1995 
as 
$$
\lambda =  0.30366 30028 98732 65859 74481 21901 55623 \ldots
$$
(and in \cite{Daude} there is a similar estimate disagreeing only in the last place).
 In 2003, Briggs presented a value which he stated was  ``probably accurate to 468 decimals.''\cite{Briggs}.
 In 2026, Nisoli estimated $\lambda$ rigorously to 175 places.

The purpose of this note is to show that the transfer operator approach of Mayer-Roepstorff can be adapted, and combined with a Lagrange-Chebychev interpolation method and the use of Hurwicz zeta functions,   to give accurate estimates on $\lambda$ which are easily validated.

\section{The transfer operator}
The problem can be rewritten in terms of the {\it Gauss map} (or {\it continued fraction transformation})
$T: [0,1] \to [0,1]$ defined by
$$
T(x) = 
\begin{cases}
\left\{\frac{1}{x}\right\} &\hbox{ if } 0 <  x \leq 1 \cr
0 &\hbox{ if } x=0 \cr
\end{cases}
$$
and then 
$F_n(x) = \ell \left(T^{-n}[0,1] \right)$.
\footnote{In \cite{MM} an interesting extension to finite group skew products is proved.}

We next  introduce a convenient Banach space based one one originally described   by Mayer and Roepstorff \cite{MR2}. Let $ D = \{z \in \mathbb C \hbox{ : } |z-1| < \frac{3}{2}\}$ and denote 
the closed disk $ \overline D = \{z \in \mathbb C \hbox{ : } |z-1| \leq  \frac{3}{2}\}$.
\begin{defn}
Let $\mathcal B$ be the Banach space of analytic functions $f: D \to \mathbb C$ 
for which
\begin{enumerate}
\item If $x \in [0, 1]$ then $f(x) \in \mathbb R$; and 
\item
 the function $f(z)$ and the derivative $f'(z)$ extend as continuous  bounded functions to 
  $\overline D$
  \end{enumerate}
The norm on  $\mathcal B$ is taken to be
$$
\|f \| = \max\left\{
\sup_{z \in \overline D} |f(z)|,  \quad\sup_{z \in \overline D} |f'(z)|
\right\}.
$$
\end{defn}
The analytic functions on $D$ which extend continuously to the boundary are a more familiar Banach space with respect to the supremum norm.  However, $\mathcal B$ is seen to be a closed subspace of this space.

We can now recall  a standard linear operator.

\begin{defn}
The {\it  transfer operator}  $\mathcal L : \mathcal B \to \mathcal B$ 
associated to  the Gauss map  is the bounded linear operator
given by
$$
\mathcal L w(x) = \sum_{n=1}^\infty \frac{1}{(x+n)^2} w\left(\frac{1}{n+1} \right).
$$
\end{defn}

\noindent 
The following properties are well known \cite{Mayer}.
\begin{enumerate}
\item
The operator $\mathcal L:  \mathcal B \to \mathcal B$ is a compact 
operator (in fact, nuclear) and has isolated eigenvalues accumulating to zero.
\item
There is a  unique $T$-invariant  probability measure $\mu$
(usually called  the 
Gauss measure) 
 equivalent to Lebesgue measure $\ell$
 with density  $\frac{d\mu}{d\ell} = \frac{1}{\log 2 (1+x)}$,  occurring as the eigenvector for the simple maximal eigenvalue $1$.
 \item
 The second eigenvalue $\lambda_2$ 
(ordered by decreasing modulus) is also  simple and isolated. 
 \end{enumerate}
 
   Moreover, since by the change of variable formula one can write
$ \int_0^1 (\mathcal L^n \chi_{[0,x]})(t) dt =  \int_0^1 \chi_{[0,x]}dt$ for $n \geq 0$.
We can then write 
$$
F_n(x) = \int_0^1 (\mathcal L^n \chi_{[0,x]})(t) dt \to \frac{1}{\log 2} \int_0^x \frac{du}{1+u}
= \frac{\log (1+x)}{\log 2}
$$
(i.e., the Gauss theorem).
The essence of Wirsing's  approach was that the  exponential speed of convergence $-\lambda$ corresponds to the second eigenvalue $\lambda_2$ (in modulus) for $\mathcal L$, whose value we would therefore like to accurately estimate. 

\section{Properties of the transfer operator}
Since we are modifying the definitions of the spaces and cones from \cite{MR2} we need to check that the results we require still hold.
\subsection{A function space}
We want to further restrict to the subspace
$$
\mathcal B_{\mathbb R} = \left\{
f \in \mathcal B \hbox{ : } f(x) \in \mathbb R \hbox{ for } x \in [0,1] \hbox{ and } \int_0^1 f d\mu=0
\right\}
$$
of functions  which map $[0,1]$ to the real line. \footnote{Formally this definition appears  different from that in \cite{MR2} although the spaces are actually the same}
By the definition we can see that $\mathcal L: \mathcal B_{\mathbb R} \to \mathcal B_{\mathbb R}$.

\subsection{A cone}
We can now consider the cone
$$
C = \left\{ f \in \mathcal B_{\mathbb R} \hbox{ : } f'(x) \geq 0 \hbox{ for } 0 \leq x \leq 1 \right\}
$$

We now want to establish some properties of the cone.

\begin{lem} $C$ is a proper cone (i.e.,  if $f \in \mathcal C$ and $- f \in \mathcal C $ then $f=0$).
\end{lem}

\begin{proof}
Since  $f \in C$ we have  that both $f'(x), -f'(x) \leq 0$ when $0 \leq x \leq 1$ and thus $f'(x) = 0$, which implies that $f: D_{\mathbb R} \to \mathbb R$ is a constant function.
Moreover, since $ \int_0^1 f d\mu=0$ we can conclude $f=0$.
\end{proof}

\begin{lem} $C$ is reproducing  (i.e.,  if $f \in \mathcal B_{\mathbb R}$  then we can choose $f_1, f_2 \in C$ such that $f = f_1 - f_2$).
\end{lem}

\begin{proof}
Observe that $w(z) = z-1 \in \mathcal B_{\mathbb R}$ since 
$w'(x) = 1$ and $\int w d\mu=0$.
Given any $f \in \mathcal B_{\mathbb R}$ let 
$$M = \sup_{0 \leq x \leq 1}|f'(x)| < +\infty.$$
Since $C$ is a cone we have that $f_1(z) = M w(z)$.  Moreover, $f_2(z) = f(z) + f_1(z)$ satisfies
$$
f_2'(x) = f'(x) + f_1'(x) = f'(x) + M  \geq 0 \hbox{ for } 0 \leq x \leq 1
$$
and so $f_2 \in C$.  Finally, by construction $f(z) = f_1(z) - f_2(z)$.
\end{proof}

\begin{lem} $C$ has non-empty interior
\end{lem}
\begin{proof}
We need only observe that $w \in \hbox{int}(C)$. 
\end{proof}

\subsection{An operator}
It is convenient to  conjugate the operator 
$\mathcal L$ by multiplication by $(x+1)$ 
to get 
the linear operator $\mathcal M : \mathcal B_{\mathbb R} \to \mathcal B_{\mathbb R} $ of the form 
$$\mathcal M w(x) = (x+1) \mathcal L\left(\frac{w(x)}{x+1} \right)
\hbox{ for } w \in \mathcal B_{\mathbb R} 
$$ which can be rewritten as  
$$
\mathcal M w(x) = (x+1) \sum_{n=1}^\infty \frac{1}{(x+n)(x+n+1)} w\left(\frac{1}{x+1} \right).
$$
Although the operators have the same spectrum, the advantage of studying $\mathcal M$ is that it preserves the constants.

\begin{lem}
The operator  $\mathcal M : \mathcal B_{\mathbb R} \to \mathcal B_{\mathbb R} $ preserves the cone $C$.
\end{lem}

\begin{proof}
By definition, for $f \in C$ the image under $-\mathcal M$  takes the form
$$
-\mathcal M f(x) = - \sum_{n=1}^\infty \frac{x+1}{(x+n)(x+n+1)}  f \left( \frac{1}{x+n} \right)
\hbox{ for } 0 \leq x \leq 1.
$$
Differentiating in $x$ gives
$$
\begin{aligned}
(-\mathcal M f)'(x) &= \sum_{n=1}^\infty \frac{1}{(x+n)(x+n+1)} \left(\frac{x+1}{x+n} + \frac{x+1}{x+n+1}-1 \right) f\left( \frac{1}{x+n}\right)\cr
&\qquad+ \sum_{n=1}^\infty  \frac{1}{(x+n)^3(x+n+1)} 
\underbrace{ f'\left( \frac{1}{x+n}\right)}_{\geq 0}\cr
\end{aligned}
\eqno(3)
$$ 
and we see that the second term in (3) is positive.

It remains to consider the first term in (3) 
For $0 \leq x \leq 1$ we can choose $n(x)$ with 
$$
 \frac{1}{(x+n)(x+n+1)} \left(\frac{x+1}{x+n} + \frac{x+1}{x+n+1}-1 \right)
 \begin{cases}
 \geq 0 &\hbox{ for } n \leq n(x) \cr
  <  0 &\hbox{ for } n>  n(x)
 \end{cases}
$$
We can then write 
$$
\begin{aligned}
&\sum_{n=1}^\infty \frac{1}{(x+n)(x+n+1)} \left(\frac{x+1}{x+n} + \frac{x+1}{x+n+1}-1 \right) f\left( \frac{1}{x+n}\right)
\cr
&= \sum_{n=1}^{n(x)} 
\underbrace{\frac{1}{(x+n)(x+n+1)} \left(\frac{x+1}{x+n} + \frac{x+1}{x+n+1}-1 \right)}_{\geq 0} \underbrace{f\left( \frac{1}{x+n}\right)}_{\geq f\left( \frac{1}{x+n(x)}\right)}
\cr
&+ \sum_{n=n(x) + 1}^\infty 
\underbrace{\frac{1}{(x+n)(x+n+1)} \left(\frac{x+1}{x+n} + \frac{x+1}{x+n+1}-1 \right)}_{\leq 0} 
\underbrace{f\left( \frac{1}{x+n}\right)}_{\leq f\left( \frac{1}{x+n(x)}\right)}
\cr
\end{aligned}
\eqno(4)
$$
However, since 
$$\sum_{n=1}^\infty \frac{1}{(x+n)(x+n+1)} \left(\frac{x+1}{x+n} + \frac{x+1}{x+n+1}-1 \right) = 1$$
we see that the first (positive)  term in (4) is larger than the absolute value of the second (negative) term in (4).
\end{proof}

The cone $C$ gives a natural ordering with $f_1 \geq f_2 \iff f_1 - f_2 \in C$.

The following property is important in what follows.

\begin{defn}
We say that $- \mathcal M: \mathcal B_{\mathbb R} \to  \mathcal B_{\mathbb R}$ is $u_0$-positive (where $u_0 \in C\setminus\{0\}$) if for any $f \in C\setminus\{0\}$ there exist $\alpha \leq \beta$ and $n \geq 1$
such that $\alpha u_0 \leq (-\mathcal M)^n f \leq \beta u_0$ for all $n \geq 0$.
\end{defn}

We can now verify this property.

\begin{lem}
The operator 
$- \mathcal M: \mathcal B_{\mathbb R} \to  \mathcal B_{\mathbb R}$ is $u_0$-positive (with $n=1$). 
\end{lem}
\begin{proof}
Let $u_0(z) = z-1 \in C$ then $u_0'=1$.
Given $f \in \mathcal B_{\mathbb R}\setminus 0$ the image $-\mathcal M f \in \mathcal B_{\mathbb R}$ 
extends continuously to the closed disk $\overline D_{\mathbb R}$ and we can let
\footnote{On page 342 of \cite{MR2} $\alpha$ and $\beta$ get interchanged}
$$
\alpha = \inf_{0 \leq x \leq \beta}  (-\mathcal M f)'
\hbox{  and }
\beta  = \sup_{0 \leq x \leq \beta}  (-\mathcal M f)'
$$
It remains to check that $\alpha > 0$. By compactness of the interval, we can assume for a contradiction there exists $0 \leq x_0 \leq 1$ such that  $(-\mathcal M f)'(x_0) = 0$.  Taking $x=x_0$ in (3), and the positivity of the first summation, we can deduce that $f'(1/(x_0+n))=0$ for all $n \geq 1$.
\footnote{In \cite{MR2} there appears to be  typographical error where  function $f$  is written instead of  $f'$}
Since $f'(z)$ is an analytic function we deduce that it must be constant.  But since $\int_0^1 f d\mu=0$ we deduce that $f=0$, i.e., it is identically zero.    This contradicts the assumption on $f$.  

\end{proof}

The above properties are sufficient  to apply a theorem of Krasnoselskii (see  \cite{Mayer-book}, Theorem C.1, p.138) to the operator  $-\mathcal M : \mathcal B_{\mathbb R} \to \mathcal B_{\mathbb R}$. 

\begin{prop}\label{kras}
Let  $-\mathcal M : B_{\mathbb R} \to B_{\mathbb R}$ be as above.
\begin{enumerate}
\item There exists a simple  eigenvalue $-\lambda_2 >0$ for $-\mathcal M$.
\item For any $u_0 \in C\setminus \{0\}$ with  $0 < \alpha \leq \beta$
satisfying  
$$\alpha u_0 \leq (-Vu_0) \leq \beta u_0
\hbox{ then }
\alpha  \leq - \lambda_2 \leq \beta.$$
\item There is an associated  eigenvector $h_2 \in C$ unique up to a scaler; and 
\item 
all other eigenvalues have strictly smaller absolute value.   
\end{enumerate}
\end{prop}

\section{Min-max bounds on the second eigenvalue $\lambda_2$}

The results in the previous section  lead to a  useful  min-max bound for this eigenvalue $\lambda_2$ of the  following form:
\footnote{This is similar to that in \cite{Mayer2}}

\begin{lem}[after Mayer-Roepstorff (cf. \cite{MR2} equation (4.12)]\label{mayer}
Let $f \in C$  then 
$$
\min_{0 \leq x \leq 1}  
\frac{(\mathcal M f)'(x)}{f'(x)}
\leq 
|\lambda_2| \leq \max_{0 \leq x \leq 1} 
\frac{(\mathcal M f)'(x)}{f'(x)}.  \eqno(5)
$$
\end{lem}
It is the  explicit bound (5) in this lemma  which serves  to make the numerical estimates  we present later rigorous.

\begin{proof}
The proof of Lemma \ref{mayer} is sketched in Section 4 of \cite{MR2}.
The starting point  is that when one  considers the quotient space $\mathcal B/\mathbb C$ and on this quotient  space one can use the norm $\|w + \mathbb C\| =  \|w'\|_\infty$.
The operator $\mathcal M : \mathcal B \to \mathcal B$  then induces a well defined operator $V: \mathcal B/\mathbb C \to \mathcal B/\mathbb C$ on the quotient space. Moreover, the spectrum of $V$ coincides with that of $\mathcal L$ {\it except that $\lambda_1$ no longer appears 
in the spectrum of $V$}  and thus $\lambda_2 < 0$ is  now the largest eigenvalue (in modulus) for $V$.

If $f$ is any function in the  cone $C$  then so is its image $-\mathcal M f \in C$ and, moreover, 
$$
\underbrace{ \left( \inf_{0 \leq y \leq 1} \left\{  \frac{(-\mathcal M f)(y)}{f(y)}\right\} \right)}_{=:\alpha} f(x)
\leq
 (-\mathcal M f)(x)
\leq
\underbrace{\left(  \inf_{0 \leq y \leq 1} \left\{  \frac{(-\mathcal Mf)(y)}{f(y)}\right\} \right) }_{=:\beta}f(x).
$$
for any $0 \leq x \leq 1$.   We can now apply item (3) of Proposition  \ref{kras} to deduce that:
$$
\underbrace{ \left( \inf_{0 \leq y \leq 1} \left\{  \frac{(-\mathcal Mf)(y)}{f(y)}\right\} \right)}_{=:\alpha}
\leq
-\lambda_2
\leq
\underbrace{\left(  \inf_{0 \leq y \leq 1} \left\{  \frac{(-\mathcal Mf)(y)}{f(y)}\right\} \right) }_{=:\beta}.
$$

\end{proof}

\begin{example}
In (\cite{MR2}, Remark, p. 343)  they consider the particular  example of $f(z) = \frac{z+1}{z+r} - c$ with $r = 1.14617$ and 
suitable $c$.   This allows them to  deduce the rigorous bound $-0.2995 \geq \lambda_2  \geq - 0.3038$.
\end{example}

In a later section we will discuss how to  choose $f \in \mathcal B$ to be a suitable polynomial which leads to improved bounds on $\lambda_2$. 

\begin{rem}
If we used the original result of \cite{MR2} then one would obtain bounds
$$
\min_{-\frac{1}{2} \leq x \leq \leq \frac{5}{2}}  
\frac{(\mathcal M f)'(x)}{f'(x)}
\leq 
|\lambda_2| \leq \max_{-\frac{1}{2} \leq x \leq \frac{5}{2}} 
\frac{(\mathcal M f)'(x)}{f'(x)} 
$$
which leads to weaker bounds in our examples in a later section.
\end{rem}

\begin{rem}
The approach of Babenko \cite{Babenko}, Flajolet-Vall\'ee \cite{Flajolet} and Briggs \cite{Briggs} to estimating $\lambda_2$ was based on approximating the transfer operator by a matrix representing the  truncation to polynomials the action on Taylor expansions about $z=1/2$.  However, this presents problems in rigorously estimating the error term. For example, Briggs considered a $1200 \times 1200$ but with no error estimates presented \cite{Briggs}.    In \cite{MR1}, Mayer and Roepstorff  present an alternative approach by relating  $\mathcal L$ to an integral operator.  
\end{rem}
\section{Hurwicz zeta functions}

We need to understand how $\mathcal M$ acts on monomials.   
It has been observed (for example in \cite{Flajolet}) that these can be conveniently presented in terms of the Hurwicz zeta function to address the  difficulty of dealing  with the infinite summations.  
 In particular, when $w_k(x) = x^k$ (for $k \geq 0$) for  $0 \leq x \leq 1$  then 
$$
\begin{aligned}
\mathcal M w_k(x) &= (x+1)\sum_{n=1}^\infty \frac{1}{(x+n)(x+n+1)}
\left(
\frac{1}{x+n}
\right)^k\cr
&=
(x+1)
\sum_{n=1}^\infty \frac{1}{(x+n)^{k+1}(x+n+1)}\cr
\end{aligned}
$$
for $k \geq 1$ and $\mathcal Mw_0(x) = 1$.
Observe that 
$$
\begin{aligned}
 \frac{1}{(x+n)^{k+1}(x+n+1)} & =  \frac{1}{(x+n)^k} \left(\frac{1}{(x+n) (x+n+1)}\right) \cr
 &=   \frac{1}{(x+n)^k} \left( \frac{1}{x+n} - \frac{1}{x + n +1} \right)\cr
 &=  \frac{1}{(x+n)^{k+1}} - \frac{1}{(x+n)^k(x + n +1)}\cr
\end{aligned}
$$
and thus multiplying by $(x+1)$ and summing over $n \geq 1$ we can deduce that
$$
\mathcal M w_k(x) =  
(x+1)
(L(x, k+1)  - \mathcal M w_{k-1}(x)
$$s
where 
$$
L(x, k+1) =   \sum_{n=1}^\infty \frac{1}{(x+n)^{k+1}}=  \zeta(k+1, x)-  \frac{1}{x^{k+1}}
$$
and $\zeta(s, x) = \sum_{n=0}^\infty \frac{1}{(x+n)^s}$ is the Hurwicz zeta function.  
By induction in $k$ we can write
$$
\begin{aligned}
\mathcal M w_k(x) & =  
(x+1)L(x, k+1)  - \mathcal M w_{k-1}(x) \cr
\mathcal M w_{k-1}(x) & =   
(x+1)L(x, k)  -  \mathcal M w_{k-2}(x)\cr
&\qquad \vdots \cr
\mathcal M w_{2}(x) & =  (x+1)L(x, 3) 
-  \mathcal M w_{1}(x)  \cr
\mathcal M w_{1}(x) & = (x+1) L(x, 2) 
-  \underbrace{\mathcal M w_{0}(x)}_{= 1} \cr
\end{aligned}
$$
For $k \geq 3$ we  can use these identities  to write 
$$
\begin{aligned}
\mathcal M w_k(x) & =  (x+1) \left( L(x, k+1)  -  L(x, k)\right)  +  \mathcal M w_{k-2}(x) \cr
& =  (x+1) \left(L(x, k+1)  -  L(x, k)  +  L(x, k-1)\right)  - \mathcal M w_{k-3}(x)  \cr
&\qquad \ldots\cr
&= (x+1) \left(L(x, k+1)  -  L(x, k)  +  L(x, k-1)   - \cdots  + (-1)^{k-1} L(x, 2) \right) +  (-1)^{k}
\underbrace{\mathcal M w_0(x)}_{=1}
\end{aligned}
$$
We can succinctly write this as:  

\begin{lem} We can write 
$$
\mathcal M w_k(x)  = 
\begin{cases}
\sum_{t=2}^{k+1} (-1)^{k+1 - t}  L(x, t) 
+ (-1)^{k}& \hbox{ for $k \geq 1$}\cr
1 & \hbox{ for $k = 0$.}\cr
\end{cases}
$$

\end{lem}
In particular, since the Hurwicz zeta function is a function  for which the 
values can be accurately estimated
and thus   $\mathcal Mw_k(x)$
can be  efficiently estimated.  

\begin{rem}
In practice, one can estimate the Hurwicz zeta function using the Euler-Maclaurin 
 approximation or, more directly,   estimate the series $L(x, k+1)$
\end{rem}

\section{Estimates on $\lambda_2$}

Fix $m \geq 2$.
After suitable scaling to $[0,1]$, we  can choose Chebychev points $\{x_j\}_{j=1}^m \subset [0,1]$
defined by $$x_j = \frac{1}{2} + \frac{1}{2}  \cos\left(\frac{\pi (2j - 1)}{2m}\right)$$ and Lagrange polynomials 
$\{\ell_i(x)\}_{i=1}^m$ on $[0,1]$ 
defined by
$$
\ell_j(x) = 
\frac{\prod_{i} \left(x - x_j \right)}{\prod_{i \neq j} \left(x_i - x_j \right)}
$$
so that we then have that 
$$
\ell_i(x_j) = 
\begin{cases}
1 & \hbox{ if } i = j \cr
0 & \hbox{ if } i \neq  j. \cr
\end{cases}
$$
We proceed as follows.

\medskip
\noindent
{\bf Step 1.}
We can then associate the $m \times m$ matrix with entries
$$
M(i,j) = (\mathcal M \ell_i)(x_j)
\hbox{ for } 1 \leq i,j \leq m.
$$
If we write $\ell_i(x) = \sum_{r=0}^{m-1} a_r^{(i)} w_r(x)$ 
(for $1 \leq i \leq m$)
then we can rewrite
$$
(\mathcal M \ell_i)(x) =  \sum_{r=0}^{m-1} a_r^{(i)} (\mathcal M w_r)(x)
\hbox{ for } i = 1, \cdots, m
$$
and thus 
$$
M(i,j) = (\mathcal M \ell_i)(x_j) =   \sum_{r=0}^{m-1} a_r^{(i)} (\mathcal M w_r)(x_j)
\hbox{ for } 1 \leq i,j \leq m.
$$

\medskip
\noindent
{\bf Step 2.}
The second eigenvalue (in absolute value) for  $M$ will have  
a (left)  eigenvector $v = (v_i)_{i=1}^M$.   We can then choose 
$$f(x) = \sum_{i=1}^m v_i \ell_i(x).$$

\medskip
\noindent
{\bf Step 3.}
Finally, 
$$
(\mathcal M f)(x) 
= \sum_{i=1}^m v_i  (\mathcal M \ell_i) (x)
=
\sum_{i=0}^{m-1} v_i  
\left( \sum_{r=0}^{m-1} a_r^{(i)} (\mathcal M w_r)(x)\right)
$$
which we can evaluate these using the Lemma to get rigorous  accurate numerical estimates on $L(x, t)$.
We can adjust the choice of $m$ as required to improve the approximation. 

\begin{rem}
The efficiency of this approach depends on the time needed to evaluate the maximal eigenvalue and eigenvector of an $m \times m$ matrix.  A practical issue is to determine the minimum $\alpha$
and the maximum $\beta$ in (5) which benefits from 
 the functions $f$ and $\mathcal M f$ being analytic.
\end{rem}

\begin{example}[$m=6$] 
When $m=6$ the above approach suggests a degree $5$ polynomial:
$$
\begin{aligned}
f(x) = 0.619273
- 
 2.004011
 x + 
 1.415702
 x^2 - 
 0.883438
 x^3 + 
 0.404095
 x^4 - 
 0.091799
  x^5
 \end{aligned}
$$
\begin{figure}[h!]
\centerline{
\includegraphics[width=0.38\textwidth]{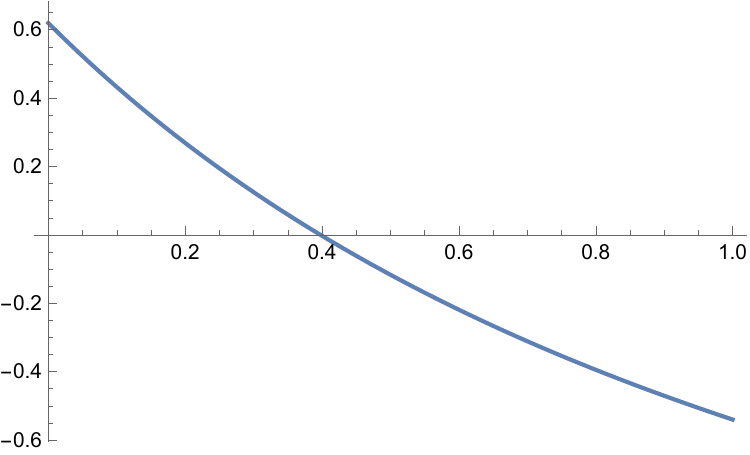} \hskip 1cm
\includegraphics[width=0.38 \textwidth]{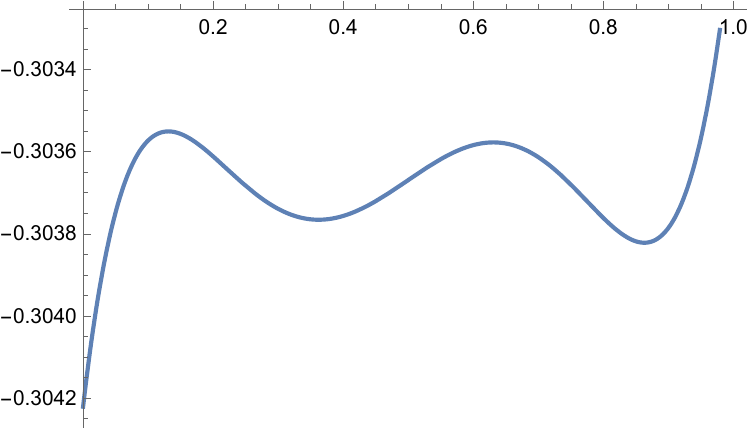}}
\caption{For example 6.2 (a) The polynomial function $f(x)$; (b) The function $(Vf)'(x)/f'(x)$}
\end{figure}
Then (5)  leads to bounds 
$$
-0.304220
\leq \lambda_2 \leq 
-0.303033
$$
accurate to $2$ decimal places.
\end{example}

\begin{example}[$m=12$] 
When $m=12$ the above approach suggests a degree $11$ polynomial:
$$
\begin{aligned}
f(x)&=0.437917
-  1.418027
 x + 
 1.014254
 x^2 - 
 0.694128
 x^3 + 
 0.457471
 x^4 - 
 0.291319
 x^5 \cr 
  &\quad + 
 0.178402
 x^6 - 
 0.102401
  x^7 + 
 0.051754
 x^8 - 
 0.020695
  x^9 + 
 0.005557
 x^{10} - 
 0.000725
 x^{11}
 \end{aligned}
$$
\begin{figure}[h!]
\centerline{
\includegraphics[width=0.38\textwidth]{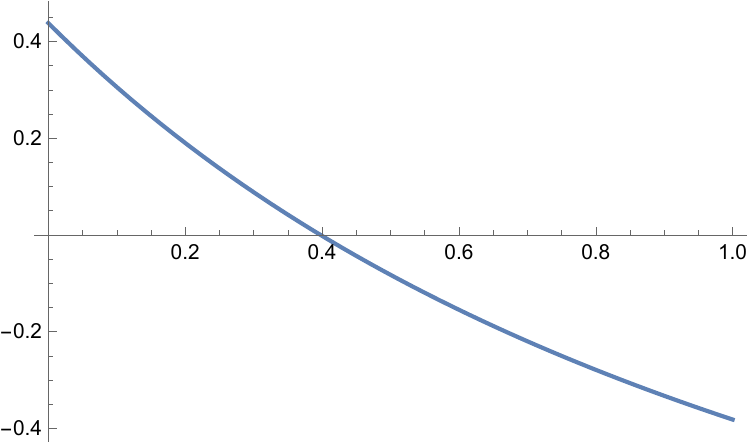} \hskip 1cm
\includegraphics[width=0.38 \textwidth]{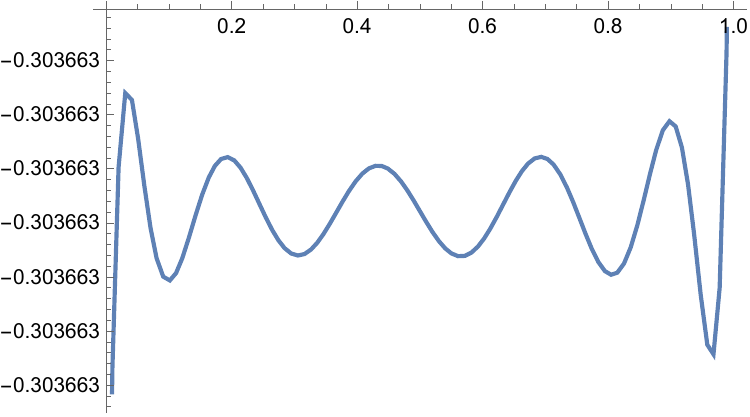}
}
\caption{For example 6.3 (a) The polynomial function $f(x)$; (b) The function $(Vf)'(x)/f'(x)$}
\end{figure}
Then (5)  leads to bounds 
$$
\begin{aligned}
&-0.3036630081
 \leq \lambda_2 \leq 
-0.3036629966
\end{aligned}
$$
accurate to $4$ decimal places.
\end{example}

\begin{example}[$m=50$] 
When $m=80$ the above approach suggests a degree $49$ polynomial.
Then (5)  leads to bounds 
$$
\begin{aligned}
&-0.303663002898732658597448121901556233110877352253659370 \cr
&\qquad \leq \lambda_2 \leq 
-0.303663002898732658597448121901556233110877352253656714
\end{aligned}
$$
accurate to $51$ decimal places.
\end{example}

\begin{example}[$m=80$] 
When $m=80$ the above approach suggests a degree $79$ polynomial.
Then (5)  leads to bounds 
$$
\begin{aligned}
&-0.3036630028987326585974481219015562331108773522536578951882454814671440498 \cr
&\qquad \leq \lambda_2 \leq 
-0.30366300289873265859744812190155623311087735225365789518824548146730843942
\end{aligned}
$$
accurate to $67$ decimal places.
\end{example}

Further improvements would follow by choosing larger values of $m$.

\end{document}